\documentclass[12pt]{amsart}

\usepackage{amsrefs}
\usepackage{enumerate}
\usepackage{hyperref}
\hypersetup{hidelinks,colorlinks}
\usepackage{amsmath}
\usepackage{amssymb}
\usepackage{mathtools}
\usepackage{mathrsfs}
\usepackage{color}

\usepackage[normalem]{ulem}
\usepackage{comment}
\usepackage{nicefrac}

\hypersetup{
    linkcolor=red,
    citecolor=blue,
 }

\def\Xint#1{\mathchoice
{\XXint\displaystyle\textstyle{#1}}%
{\XXint\textstyle\scriptstyle{#1}}%
{\XXint\scriptstyle\scriptscriptstyle{#1}}%
{\XXint\scriptscriptstyle\scriptscriptstyle{#1}}%
\!\int}
\def\XXint#1#2#3{{\setbox0=\hbox{$#1{#2#3}{\int}$ }
\vcenter{\hbox{$#2#3$ }}\kern-.6\wd0}}

\def\dashint{\Xint-}

\usepackage{xcolor}

\newtheorem{theorem}{Theorem}

\newtheorem{lemma}{Lemma}

\newtheorem{remark}{Remark}
\newtheorem{definition}{Definition}

\newcommand{\RR}{\mathbb{R}}
\newcommand{\NN}{\mathbb{N}}

\newcommand{\Ei}{\text{\normalfont Ei}}
\newcommand{\li}{\text{\normalfont li}}

\newsavebox\foobox
\newlength{\foodim}

\newcommand{\defeq}{\vcentcolon=}

\usepackage{graphicx}

\title[Proving the Puiseux representation of the exponential integral]{An alternative proof of the Puiseux representation of the exponential integral}
\author{Glenn Bruda}
\thanks{\textit{Email:} \texttt{glenn.bruda@ufl.edu} (Glenn Bruda)}
\date{\today}

\begin{document}

\begin{abstract}
     Working from definitions and an elementarily obtained integral formula for the Euler-Mascheroni constant, we give an alternative proof of the classical Puiseux representation of the exponential integral. 
\end{abstract}

\maketitle

\section{Introduction}

The exponential integral, denoted $\Ei$, is a nonelementary\footnote{It may be shown that $\Ei$ is nonelementary via the Risch algorithm (see \cite{risch}*{Chapter 12} for an introduction on such).} special function which satisfies
\begin{align}
    \Ei(x)-\Ei(1)=\int_{1}^{x}\frac{e^t}{t}dt,
\end{align}
where $\Ei(1)\approx1.895$. More precisely, as a consequence of Theorem \ref{thm1},
\begin{align}
    \Ei(1)=\gamma+\sum_{n\geq1}\frac{1}{n(n!)},
\end{align}
where $\gamma$ is the Euler-Mascheroni constant.

The exponential integral finds extensive applications in physics and engineering. By \cite[\href{https://dlmf.nist.gov/7.5\#E13}{Equation 7.5.13}]{NIST:DLMF}, it may also be used to compute the Goodwin-Staton integral:
\begin{align}
    \int_{0}^{\infty}\frac{e^{-t^2}}{t+x}dt=e^{-x^2}\sqrt{\pi}\int_{0}^{x}e^{t^2}dt-\frac{1}{2}e^{-x^2}\Ei(x^2).
\end{align}
Furthermore, the salient logarithmic integral
\begin{align}
    \li(x)\defeq\int_{\mu}^{x}\frac{1}{\log t}dt
\end{align}
relates to the exponential integral by
\begin{align}
    \li(x)=\Ei(\log x),
\end{align}
where $\mu\approx1.451$ is the Ramanujan–Soldner constant (see \cite{soldner}).

To avoid a circular definition, we give an alternative characterization of the exponential integral using the Cauchy principal value, which assigns finite values to divergent integrals. Defining $\overline{\RR}\defeq\RR\cup\{+\infty,-\infty\}$ as the extended real number line, we give the following definition.

\begin{definition}[Cauchy principal value]
    Let $a,b\in\overline{\RR}$. For a function $f$ with a singularity at $c\in(a,b)$, we define
    \begin{align}
        \dashint_{a}^{b}f(t)dt\defeq\lim_{\varepsilon\to 0^+}
        \left(\int_{c+\varepsilon}^{b}f(t)dt+\int_{a}^{c-\varepsilon}f(t)dt\right).
    \end{align}
\end{definition}

Defining $\RR^{\times}\defeq\RR\backslash\{0\}$, the proper definition of the exponential integral is now given using a Cauchy principal value.

\begin{definition}[Exponential integral]
    We define $\Ei:\RR^{\times}\to\RR$ by
    \begin{align}
        \Ei(x)\defeq\dashint_{-\infty}^{x}\frac{e^t}{t}dt.
    \end{align}
\end{definition}

\begin{remark}
       When $x<0$, the Cauchy principal value $\Ei(x)$ coincides with the standard convergent integral since the limits of integration do not pass the singularity at the origin; that is, for $x<0$,
    \begin{align}
         \Ei(x)=\int_{-\infty}^{x}\frac{e^t}{t}dt.
    \end{align}
\end{remark}

A classical result on the exponential integral is that it admits the Puiseux representation (or Puiseux series\footnote{A Puiseux series is a power series which may contain fractional exponents and nested logarithms.})
\begin{align}
    \Ei(x)=\gamma+\log|x|+\sum_{n\geq1}\frac{x^n}{n(n!)}\label{puiseuxrep}
\end{align}
for any $x\in\RR^{\times}$. This form is quite useful for computing values of $\Ei$ since the power series converges so rapidly.

Since the integrand $e^t/t$ has a simple pole at $t=0$, the appearance of the logarithm is not surprising. Equally unsurprising is the power series in \eqref{puiseuxrep}, which is simply the antiderivative of the power series of $(e^t-1)/t$. The non-obvious aspect of this Puiseux representation is the presence of the Euler-Mascheroni constant; indeed, most of the work done to obtain $\eqref{puiseuxrep}$ involves determining why this constant appears.

The standard proof of \eqref{puiseuxrep}, e.g., \cite[pg.635]{standardproof}, involves calculating $\Gamma'(1)=-\gamma$, where $\Gamma$ is Euler's gamma function. In this article, we present an alternative proof of \eqref{puiseuxrep} using instead the elementarily obtained integral formula
\begin{align}\label{integralformula}
    \int_{0}^{1}\frac{1-e^{-t}-e^{-1/t}}{t}dt=\gamma
\end{align}
seen in \cite[pg.243]{EMintegral}. A thorough discussion expounding on the outline of the proof of \eqref{integralformula} given in \cite[pgs.236,243]{EMintegral} may be found in \cite{discussionofproof}.

\section{Proof}

First, we introduce a series formula for $\int_{1}^{x}f(t)dt$ which may be obtained simply by infinite integration by parts.

\begin{definition}[Harmonic numbers]
    For $n\in\NN$, we define $H_n\defeq\sum_{k=1}^{n}1/k$.
\end{definition}

\begin{lemma}\label{lem1}
    Suppose $f$ is analytic at $x\in\RR$ and $1$. Then
    \begin{align}
        \int_{1}^{x}f(t)dt=\sum_{n\geq0}(-1)^n g^{(n)}(x)\left(\frac{x^n}{n!}\left(\log|x|-H_n\right)\right)+\sum_{n\geq0}\frac{(-1)^n H_n g^{(n)}(1)}{n!},
    \end{align}
    where $g(x)=xf(x)$.
\end{lemma}

\begin{proof}
    Apply integration by parts repeatedly to
    \begin{align}
        \int_{1}^{x}f(t)dt=\int_{1}^{x}\frac{t f(t)}{t}dt,
    \end{align}
    integrating $1/t$ and differentiating $tf(t)$. Note that the $n$-fold antidifferentiation of $1/x$ yields
    \begin{align}
        \frac{x^n}{n!}\left(\log|x|-H_n\right),
    \end{align}
    which may be verified inductively. Since $f$ is analytic at $x\in\RR$ and $1$, the series converge.
\end{proof}

\begin{lemma}\label{lem2}
    We have
    \begin{align}
        \int_{0}^{1}\frac{1-e^{1-t}}{1-t}dt=\gamma-\text{\normalfont Ei}(1).
    \end{align}
\end{lemma}

\begin{proof}
    By \cite[pg.243]{EMintegral},
\begin{align}\label{EMalgebra}
    \gamma&=\int_{0}^{1}\frac{1-e^{-t}-e^{-1/t}}{t}dt=\int_{0}^{1}\frac{1-e^{-t}}{t}dt-\int_{0}^{1}\frac{e^{-1/t}}{t}dt\nonumber\\
    &=\int_{0}^{1}\frac{1-e^{-t}}{t}dt-\int_{1}^{\infty}\frac{e^{-t}}{t}dt=\int_{0}^{1}\frac{1-e^{-t}}{t}dt+\Ei(-1).
\end{align}
Then by \eqref{EMalgebra},
\begin{align}
    \gamma-\int_{0}^{1}\frac{1-e^t}{t}dt&=\int_{0}^{1}\frac{1-e^{-t}}{t}dt-\int_{0}^{1}\frac{1-e^{t}}{t}dt+\Ei(-1)\nonumber\\
    &=\int_{0}^{1}\frac{e^t-e^{-t}}{t}dt+\Ei(-1)\nonumber\\
        &=\lim_{\varepsilon\to 0^+}
        \left(\int_{\varepsilon}^{1}\frac{e^t}{t}dt-\int_{\varepsilon}^{1}\frac{e^{-t}}{t}dt\right)+\Ei(-1)\nonumber\\
        &=\lim_{\varepsilon\to 0^+}
        \left(\int_{\varepsilon}^{1}\frac{e^t}{t}dt+\int_{-1}^{-\varepsilon}\frac{e^{t}}{t}dt\right)+\Ei(-1)\nonumber\\
        &=\dashint_{-1}^{1}\frac{e^t}{t}dt+\Ei(-1)\nonumber\\
        &=\dashint_{-\infty}^{1}\frac{e^t}{t}dt-\int_{-\infty}^{-1}\frac{e^t}{t}dt+\Ei(-1)\nonumber\\
        &=\Ei(1)-\Ei(-1)+\Ei(-1)=\Ei(1).
\end{align}
    Consequently,
    \begin{align}
        \int_{0}^{1}\frac{1-e^{1-t}}{1-t}dt=\int_{0}^{1}\frac{1-e^t}{t}dt=\gamma-\Ei(1),
    \end{align}
    as desired.
\end{proof}

\begin{theorem}\label{thm1}
    The exponential integral yields the Puiseux representation
    \begin{align}
        \Ei(x)=\gamma+\log|x|+\sum_{n\geq1}\frac{x^n}{n(n!)},
    \end{align}
    for any $x\in\RR^{\times}$.
\end{theorem}

\begin{proof}
    Noting that the exponential function is entire, by Lemma \ref{lem1},
    \begin{align}\label{eiseriescalc}
        \Ei(x)-\Ei(1)=\dashint_{-\infty}^{x}\frac{e^t}{t}dt-\dashint_{-\infty}^{1}\frac{e^t}{t}dt=\int_{1}^{x}\frac{e^t}{t}dt\nonumber\\
        =\sum_{n\geq0}(-1)^n e^x\left(\frac{x^n}{n!}\left(\log|x|-H_n\right)\right)+e\sum_{n\geq0}\frac{(-1)^n H_n}{n!}\nonumber\\
        =\log|x|+e^x\sum_{n\geq0}\frac{(-1)^{n+1} H_n x^n}{n!}+e\sum_{n\geq0}\frac{(-1)^n H_n}{n!}.
    \end{align}
    By a special case of \cite[Corollary 3]{binomialharmonicsum},
    \begin{align}
        \sum_{k=0}^{n}{n\choose k}H_k(-1)^{k+1}=\frac{1}{n}
    \end{align}
    for $n>0$.\footnote{This relation may be alternatively verified using the integral form of the Harmonic numbers \cite{integralformofharmonic}, the binomial coefficients' recurrence relation \cite{binomialrecurrence}, or binomial inversion \cite{binomialinversion}.} Thus, by convolution of coefficients,
    \begin{align}\label{convolution}
        e^x\sum_{n\geq0}\frac{(-1)^{n+1} H_n x^n}{n!}=\left(\sum_{n\geq0}\frac{x^n}{n!}\right)\left(\sum_{n\geq0}\frac{(-1)^{n+1} H_n x^n}{n!}\right)=\sum_{n\geq1}\frac{x^n}{n(n!)}.
    \end{align}
    Putting \eqref{eiseriescalc} and \eqref{convolution} together, we obtain
    \begin{align}\label{penultimate}
        \Ei(x)=\log|x|+\sum_{n\geq1}\frac{x^n}{n(n!)}+\Ei(1)+e\sum_{n\geq0}\frac{(-1)^n H_n}{n!}.
    \end{align}
    Lastly, using the usual integral form of the Harmonic numbers, we have
    \begin{align}\label{interchange}
        e\sum_{n\geq0}\frac{(-1)^n H_n}{n!}&=e\sum_{n\geq0}\frac{(-1)^n}{n!}\int_{0}^{1}\frac{1-t^n}{1-t}dt=e\int_{0}^{1}\sum_{n\geq0}\frac{(-1)^n}{n!}\frac{1-t^n}{1-t}dt\nonumber\\
        &=e\int_{0}^{1}\frac{e^{-1}-e^{-t}}{1-t}dt=\int_{0}^{1}\frac{1-e^{1-t}}{1-t}dt=\gamma-\text{\normalfont Ei}(1)
    \end{align}
   by Lemma \ref{lem2}. The interchange of the summation and integral in \eqref{interchange} is justified since
    \begin{align}
        \int_{0}^{1}\sum_{n\geq0}\left|\frac{(-1)^n}{n!}\frac{1-t^n}{1-t}\right|dt=\int_{0}^{1}\sum_{n\geq0}\frac{1-t^n}{n!(1-t)}dt=\int_{0}^{1}\frac{e-e^{t}}{1-t}dt<\infty.
    \end{align}
    Therefore
    \begin{align}
        \Ei(x)=\gamma+\log|x|+\sum_{n\geq1}\frac{x^n}{n(n!)},
    \end{align}
    as desired.
\end{proof}

\bibliography{main}{}
\bibliographystyle{plain}

\end{document}